\documentclass{amsart}
\usepackage[utf8x]{inputenc}
\usepackage[colorlinks=true, pdfstartview=FitV, linkcolor=purple, citecolor=purple]{hyperref}
\usepackage{amsmath}
\usepackage{amsthm}
\usepackage[foot]{amsaddr}
\usepackage{amssymb}
\usepackage{academicons}
\usepackage{ dsfont }
\usepackage{lineno}
\usepackage{cleveref}
\usepackage{bbm}
\usepackage{pictex}
\usepackage{tikz}
\usepackage{pgfplots}
\pgfplotsset{compat = newest}
\usepackage{subfigure}
\usetikzlibrary{shapes, positioning,arrows, arrows.meta,calc}
\usepackage[margin=0.8in]{geometry}
\usepackage{tabularx}
\usepackage{float}
\usepackage{import}
\usepackage{enumitem}
\usepackage{listings}
\usepackage{xcolor} 

\newtheorem{theorem}{Theorem}[section]
\newtheorem{lemma}[theorem]{Lemma}
\newtheorem{proposition}[theorem]{Proposition}

\newtheorem{remark}[theorem]{Remark}

\newtheorem{example}{Example}
\newtheorem{conjecture}{Conjecture}

\title{Coincidence of invariant measure for the alternate base transformations}

\date{March 2026}
\author{Karma Dajani}
\author{Niels Langeveld}
\begin{document}

\maketitle
\begin{abstract}
    We characterize all pairs $(\beta,n),(\beta^\prime,m)$ such that the alternate $(\beta,n)$ and $(\beta^\prime,m)$-transformations $K_{(\beta,n)}$ and $K_{(\beta^\prime,m)}$ have the same absolutely continuous invariant measure, where $K_{(\beta,n)}(i,x)=(i+1 \mod 2 ,T_i(x))$ with $i\in\{0,1\}$, $T_0(x)=T_\beta (x)=\beta x \mod 1$, $T_1(x)=T_n(x)=nx\mod 1$ with $\beta>1$ real and $n\geq 2$ an integer.
    
\end{abstract}
\section{Introduction}
Given a sequence ${\bf\beta}=(\beta_0,\beta_1,\ldots)$ of real numbers with $\beta_i>1$ and $\prod_{i=0}^{\infty}\beta_i=\infty$. A Cantor expansion of $x\in [0,1)$ in base ${\bf\beta}=(\beta_0,\beta_1,\ldots)$ is an expansion of the form
\[
x=\sum_{n=0}^\infty  \frac{a_n}{\prod_{k=0}^n \beta_k},
\]
with $a_n\in\{0,\ldots ,\lceil \beta_k\rceil -1\}$. In case $a_n$ is the largest possible, then one speaks of the greedy Cantor base expansion. These are generalization of the classical $\beta$-expansion introduced by R\'enyi ~\cite{R57} and later Parry ~\cite{P60}.
Cantor base expansions been studied recently in ~\cite{CC21}, and a characterization of the associated shift spaces was given. In particular, a generalization of Parry's theorem characterizing greedy $\beta$-expansions was proved. In the special case when the base sequence 
${\bf\beta}=(\beta_0,\beta_1,\ldots,\beta_{p-1},\beta_0,\beta_1,\ldots,\beta_{p-1},\ldots)$ is periodic, which we simply denote by $(\beta_0,\beta_1,\ldots,\beta_{p-1})$, one speaks of the alternate base expansion. The symbolic and dynamical properties of such expansions were studied extensively in the past five years; see ~\cite{CC21, CCMP23, C23, MPS23, CCK24, CCD23}.

Given a periodic base $(\beta_0,\beta_1,\ldots,\beta_{p-1})$, in ~\cite{CCD23} a map $K_{(\beta_0,\ldots,\beta_{p-1})}$ was introduced that generates all alternate greedy Cantor base expansions of the form 
\[
x=\sum_{n=0}^\infty  \frac{a_n}{\prod_{k=0}^n \beta_k},
\]
where $\beta_k=\beta_{k\mod p}$ and $a_n\in\{0,\ldots ,\lceil \beta_k\rceil -1\}$. The map 
$K_{(\beta_0,\ldots,\beta_{p-1})}$ is defined on $\{0,\ldots,p-1\}\times[0,1)$ and is given by 
\[
K_{\beta_0,\ldots,\beta_{p-1}}(i,x)=(i+1\mod p , T_{\beta_i}(x)),
\]
see Figure \ref{fig:graphicalpresentation} for a graphical presentation.
\medskip
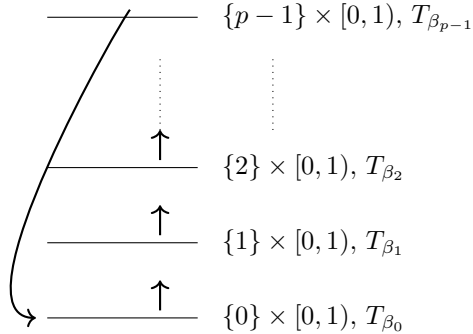
\begin{figure}[thb]\label{fig:graphicalpresentation}
\begin{center}
\begin{tikzpicture}
 \begin{scope}
  \draw (0,0) --(2,0) (0,1);
    \draw[->, thick] (1.5,.1) --(1.5,.5);
  \end{scope}
   \draw (0,1) --(2,1) (1,1);
     \draw[->, thick] (1.5,1.1) --(1.5,1.5);
 \draw (0,2) --(2,2);
     \draw[->, thick] (1.5,2.1) --(1.5,2.5);

\draw[dotted] (1.5,2.5)--(1.5,3.5) ;
 \draw (0,4) --(2,4);
     \draw[->, thick] (1.1,4.1) .. controls (1,4) and (-1.3,0) .. (-.15,0);
     \node[right] at (2.2,0) {$\{0\}\times [0,1), \, T_{\beta_0}$};
\node[right] at (2.2,1)  {$\{1\}\times [0,1),\, T_{\beta_1}$};
\node[right] at (2.2,2)  {$\{2\}\times [0,1),\, T_{\beta_2}$};
\draw[dotted] (3,2.5)--(3,3.5);
\node[right] at (2.2,4)  {$\{p-1\}\times [0,1),\, T_{\beta_{p-1}}$};
\end{tikzpicture}
\end{center}
\caption{The map $K_{(\beta_0,\ldots,\beta_{p-1})}$.}
\end{figure}
\medskip

The absolutely continuous invariant measure $\mu_{(\beta_0,\ldots, \beta_{p-1})}$ for $K_{(\beta_0,\ldots,\beta_{p-1})}$ is given by 
\[
\mu_{(\beta_0,\ldots, \beta_{p-1})}\left(\bigcup_{i=0}^{p-1}\{i\}\times A_i\right)=\frac{1}{p}\sum_{i=0}^{p-1} \mu_{\beta_{i-1}\circ\cdots\circ \beta_0\circ \beta_{p-1}\circ \cdots \beta_{i+1}\circ \beta_i}(A_i),
\]
where 
$\mu_{\beta_{i-1}\circ\cdots\circ \beta_0\circ \beta_{p-1}\circ \cdots \beta_{i+1}\circ \beta_i}$
is the absolutely continuous invariant measure for the composition 
\[T_{\beta_{i-1}}\circ\cdots\circ T_{\beta_0}\circ T_{\beta_{p-1}}\circ \cdots \circ T_{\beta_{i+1}}\circ T_{\beta_i}\] 
defined on the $i^{\text{th}}$ level and $\beta_{-1}$ should be understood as $\beta_{p-1}$. Each such composition is a piecewise linear map with constant slope $\beta_0\beta_1\cdots\beta_{p-1}$.
When the sequence of bases is periodic with period one, then the alternate greedy Cantor base expansion is the classical greedy $\beta$-expansion as defined by R\'enyi in ~\cite{R57}, so $K_{\beta}=T_{\beta}$. This map preserves, the famous R\'enyi-Parry measure $\mu_{\beta}$ with density
\[f_{\beta}(x)=\frac{1}{C}\sum_{n=0}^{\infty}\mathds{1}_{[0,T^n_{\beta}(1))}(x),\]
where $C$ is a normalizing constant ~\cite{P60}. In ~\cite{HW25}, Huang and Wang answered a long standing conjecture of Bertrand-Mathis ~\cite{BM98} by proving the following theorem.
\begin{theorem}[\cite{HW25}]\label{th:HuangWang}
For two non-integers $\beta,\beta^\prime>1$ the R\'enyi-Parry measures coincide if and only  if $\beta$ solves the equation $\beta^2=p\beta + q$ with $p,q\in \mathbb{N}$, 
$p\geq q\geq 1$ and $\beta^\prime=\beta+1$.  
\end{theorem}

In this article, we initiate a similar investigation for the alternate base transformation $K_{(\beta,n)}$, 
with $\beta>1$ a real number and $n\geq 2 $ an integer. In this case, $K_{(\beta, n)}:\{0,1\}\times [0,1)\rightarrow \{0,1\}\times [0,1)$ is defined by $K_{(\beta,n)}(i,x)=(i+1 \mod 2 ,T_i(x))$ with $i\in\{0,1\}$, $T_0(x)=T_\beta (x)=\beta x \mod 1$, $T_1(x)=T_n(x)=nx\mod 1$.  In the special case of the pair $(\beta,n)$, the $K_{(\beta,n)}$-invariant measure $\mu_{(\beta,n)}$ is given by
\[
\mu_{(\beta,n)}\Big(\{0\}\times A_0\cup \{1\}\times A_1\Big)=\frac{1}{2}\mu_{n\circ\beta}(A_0) +\frac{1}{2}\mu_{\beta\circ n}(A_1).
\]
Our aim is to characterize all pairs $(\beta,n)$ and $(\beta^\prime,m)$ such that $\mu_{(\beta,n)}=\mu_{(\beta^\prime,m)}$, where $\beta,\beta^\prime>1$ are real numbers and $n,m\geq 2$ are integers. From the above, we see that $\mu_{(\beta,n)}=\mu_{(\beta^\prime,m)}$ if and only if
\[
\mu_{n\circ\beta}=\mu_{m\circ\beta^\prime}\quad \text{ and } \mu_{\beta\circ n}=\mu_{\beta^\prime\circ m}.
\]

We prove the following characterization.
\begin{theorem}\label{th:maintheorem}
Let $\beta,\beta^\prime>1$ be real non-integer numbers, and $n,m\geq 2$ integers. Then, 
    $\mu_{(\beta,n)}=\mu_{(\beta^\prime,m)}$ if and only if $\beta=\beta^\prime =\frac{p}{q}$, with $p,q$ relatively prime and $n,m\in\{kq: k\geq 1\}$.
\end{theorem}
Note that if $\beta,\beta^\prime$ are integers, then all four component measures, $\mu_{\beta\circ n},\mu_{\beta^\prime\circ m},\mu_{n\circ\beta},\mu_{m\circ\beta^\prime} $, are  Lebesgue measure so that $\mu_{(\beta,n)}=\mu_{(\beta^\prime,m)}$.

\section{The component compositions:}
Throughout the rest, we consider pairs $(\beta,n)$ with $\beta>1$ a non-integer real number and $n\geq 2$ an integer. In this case, the invariant measure $\mu_{(\beta,n)}$ of $K_{(\beta,n)}$  is given by
\[
\mu_{(\beta,n)}\Big(\{0\}\times A\cup \{1\}\times B\Big)=\frac{1}{2}\mu_{n\circ\beta}(A) +\frac{1}{2}\mu_{\beta\circ n}(B)
\]
with $\mu_{n\circ\beta},\,\mu_{\beta\circ n}$ the invariant measure for the compositions $T_{n\circ \beta}:=T_n\circ T_\beta$ and $T_{\beta\circ n}:=T_\beta\circ T_n$ respectively.

An easy calculation shows that $T_{n\circ \beta}(x)=n\beta \mod 1$, and $T_{\beta\circ n}$ is piecewise linear with constant slope $\beta n$ and underlying fundamental partition (of rank 1) given by
\[
\mathcal{I}_{\beta\circ n}=\left\{\left[\frac{k}{n}+\frac{l}{n\beta},\frac{k}{n}+\frac{l+1}{n\beta}\right), \left[\frac{k}{n}+\frac{\lfloor \beta \rfloor
}{n\beta},\frac{k+1}{n}\right): k=0,\ldots, n-1 \text{ and } l=0,\ldots, \lfloor \beta \rfloor -1 \right\},
\]
\\
see Figure \ref{fig:partitions}.
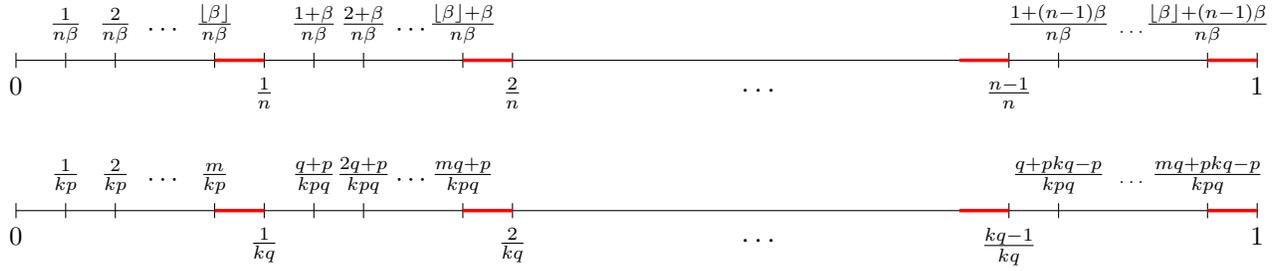
\begin{figure}[thb]
\center{\begin{tikzpicture}[xscale=1.65]
\draw(0,0)--(10,0);

\draw(0,-0.1)node[below]{$0$}--(0,0.1);
\draw(10,-0.1)node[below]{$1$}--(10,0.1);

\draw(2,-0.1)node[below]{$\frac{1}{n}$}--(2,0.1);
\draw(4,-0.1)node[below]{$\frac{2}{n}$}--(4,0.1);
\draw(8,-0.1)node[below]{$\frac{n-1}{n}$}--(8,0.1);
\node[below] at (6,-0.2){\large $\cdots$};

\draw(0.4,-0.1)--(0.4,0.1)node[above]{$\frac{1}{n\beta}$};
\draw(0.8,-0.1)--(0.8,0.1)node[above]{$\frac{2}{n\beta}$};
\draw(1.6,-0.1)--(1.6,0.1)node[above]{$\frac{\lfloor\beta \rfloor}{n\beta}$};
\node[above] at (1.2,0.2){ $\cdots$};

\draw(2.4,-0.1)--(2.4,0.1)node[above]{$\frac{1+\beta}{n\beta}$};
\draw(2.8,-0.1)--(2.8,0.1)node[above]{$\frac{2+\beta}{n\beta}$};
\draw(3.6,-0.1)--(3.6,0.1)node[above]{$\frac{\lfloor\beta \rfloor+\beta}{n\beta}$};
\node[above] at (3.2,0.2){ $\cdots$};

\draw(8.4,-0.1)--(8.4,0.1)node[above]{$\frac{1+(n-1)\beta}{n\beta}$};
\draw(9.6,-0.1)--(9.6,0.1)node[above]{$\frac{\lfloor\beta \rfloor+(n-1)\beta}{n\beta}$};
\node[above] at (9,0.2){\tiny $\cdots$};

\draw[line width=0.45mm, red ] (1.6,0)--(2,0);
\draw[line width=0.45mm, red ] (3.6,0)--(4,0);
\draw[line width=0.45mm, red ] (9.6,0)--(10,0);
\draw[line width=0.45mm, red ] (7.6,0)--(8,0);


\draw(0,-2)--(10,-2);
\draw(0,-2.1)node[below]{$0$}--(0,-1.9);
\draw(10,-2.1)node[below]{$1$}--(10,-1.9);

\draw(2,-2.1)node[below]{$\frac{1}{kq}$}--(2,-1.9);
\draw(4,-2.1)node[below]{$\frac{2}{kq}$}--(4,-1.9);
\draw(8,-2.1)node[below]{$\frac{kq-1}{kq}$}--(8,-1.9);
\node[below] at (6,-2.2){\large $\cdots$};

\draw(0.4,-2.1)--(0.4,-1.9)node[above]{$\frac{1}{kp}$};
\draw(0.8,-2.1)--(0.8,-1.9)node[above]{$\frac{2}{kp}$};
\draw(1.6,-2.1)--(1.6,-1.9)node[above]{$\frac{m}{kp}$};
\node[above] at (3.2,-1.8){ $\cdots$};

\draw(2.4,-2.1)--(2.4,-1.9)node[above]{$\frac{q+p}{kpq}$};
\draw(2.8,-2.1)--(2.8,-1.9)node[above]{$\frac{2q+p}{kpq}$};
\draw(3.6,-2.1)--(3.6,-1.9)node[above]{$\frac{mq+p}{kpq}$};
\node[above] at (1.2,-1.8){ $\cdots$};

\draw(8.4,-2.1)--(8.4,-1.9)node[above]{$\frac{q+pkq-p}{kpq}$};
\draw(9.6,-2.1)--(9.6,-1.9)node[above]{$\frac{mq+pkq-p}{kpq}$};
\node[above] at (9,-1.8){\tiny $\cdots$};

\draw[line width=0.45mm, red ] (1.6,-2)--(2,-2);
\draw[line width=0.45mm, red ] (3.6,-2)--(4,-2);
\draw[line width=0.45mm, red ] (9.6,-2)--(10,-2);
\draw[line width=0.45mm, red ] (7.6,-2)--(8,-2);

\end{tikzpicture}}
\caption{Top: $\mathcal{I}_{\beta\circ n}$ for general $\beta$, bottom: $\mathcal{I}_{\frac{p}{q}\circ kq}$  where $\lfloor\frac{p}{q}\rfloor=m$. Note that the red intervals are the only ones with a non-full branch.}
\label{fig:partitions}
\end{figure}

The map $T_{\beta\circ n}$ is defined by \
\[
T_{\beta\circ n}(x)=
\begin{cases}
    \beta n x -(k\beta +l), & x\in \left[\frac{k}{n}+\frac{l}{n\beta},\frac{k}{n}+\frac{l+1}{n\beta}\right),\  k=0,\ldots, n-1 \text{ and } l=0,\ldots, \lfloor \beta \rfloor -1 \\
    \\
    \beta n x -(k\beta +\lfloor \beta \rfloor), & x\in \left[\frac{k}{n}+\frac{\lfloor \beta \rfloor }{n\beta},\frac{k+1}{n}\right) ,\  k=0,\ldots, n-1,
\end{cases}
\]
 see Figure \ref{fig:mapscomposition} for an example. 

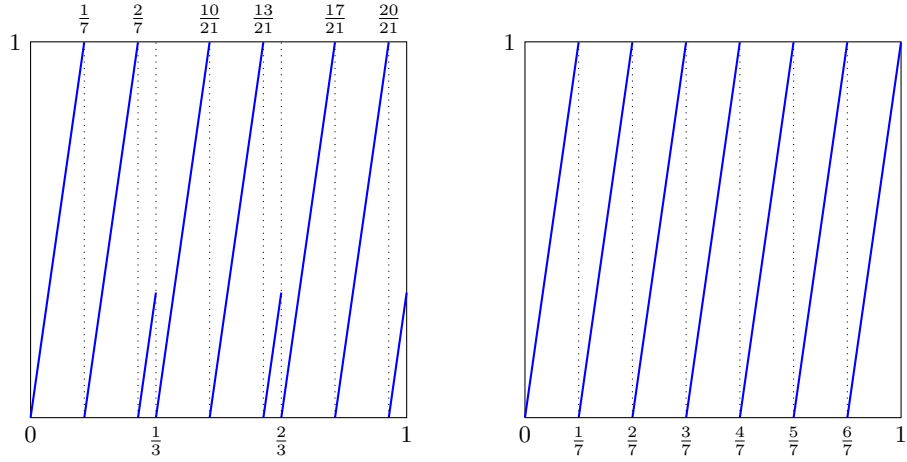
\begin{figure}[ht]
		\centering
		
		\subfigure{\begin{tikzpicture}[scale=5]
				\draw[white] (-0.25,0)--(1,0);
				\draw(0,0)node[below]{\small $0$}--(1,0)node[below]{\small $1$}--(1,1)--(0,1)node[left]{\small $1$}--(0,0);

				\draw[thick, blue, smooth, samples =20, domain=0:1/7] plot(\x,{7*\x)});
				\draw[thick,blue, smooth, samples =20, domain=1/7:2/7] plot(\x,{7*\x-1)});
                \draw[thick,blue, smooth, samples =20, domain=2/7:1/3] plot(\x,{7*\x-2)});
				\draw[thick,blue, smooth, samples =20, domain=1/3:10/21] plot(\x,{7*\x-7/3)});
		          \draw[thick,blue, smooth, samples =20, domain=10/21:13/21] plot(\x,{7*\x-7/3-1)});
                \draw[thick,blue, smooth, samples =20, domain=13/21:2/3] plot(\x,{7*\x-7/3-2)});
                \draw[thick,blue, smooth, samples =20, domain=2/3:17/21] plot(\x,{7*\x-14/3)});
                \draw[thick,blue, smooth, samples =20, domain=17/21:20/21] plot(\x,{7*\x-14/3-1)});
                \draw[thick,blue, smooth, samples =20, domain=20/21:1] plot(\x,{7*\x-14/3-2)});
                \draw[dotted](1/7,0)--(1/7,1)node[above]{\small $\frac{1}{7}$};
				\draw[dotted](2/7,0)--(2/7,1)node[above]{\small $\frac{2}{7}$};
                \draw[dotted](1/3,0)node[below]{\small $\frac{1}{3}$}--(1/3,1);
                \draw[dotted](10/21,0)--(10/21,1)node[above]{\small $\frac{10}{21}$};
                \draw[dotted](13/21,0)--(13/21,1)node[above]{\small $\frac{13}{21}$};
                \draw[dotted](2/3,0)node[below]{\small $\frac{2}{3}$}--(2/3,1);
                \draw[dotted](17/21,0)--(17/21,1)node[above]{\small $\frac{17}{21}$};
                \draw[dotted](20/21,0)--(20/21,1)node[above]{\small $\frac{20}{21}$};
		\end{tikzpicture}}
        \subfigure{\begin{tikzpicture}[scale=5]
				\draw[white] (-0.25,0)--(1,0);
				\draw(0,0)node[below]{\small $0$}--(1,0)node[below]{\small $1$}--(1,1)--(0,1)node[left]{\small $1$}--(0,0);

				\draw[thick, blue, smooth, samples =20, domain=0:1/7] plot(\x,{7*\x)});
				\draw[thick,blue, smooth, samples =20, domain=1/7:2/7] plot(\x,{7*\x-1)});
                \draw[thick,blue, smooth, samples =20, domain=2/7:3/7] plot(\x,{7*\x-2)});
				\draw[thick,blue, smooth, samples =20, domain=3/7:4/7] plot(\x,{7*\x-3)});
		          \draw[thick,blue, smooth, samples =20, domain=4/7:5/7] plot(\x,{7*\x-4)});
                \draw[thick,blue, smooth, samples =20, domain=5/7:6/7] plot(\x,{7*\x-5)});
                \draw[thick,blue, smooth, samples =20, domain=6/7:7/7] plot(\x,{7*\x-6)});
                \draw[dotted](1/7,0)node[below]{\small $\frac{1}{7}$}--(1/7,1);
				\draw[dotted](2/7,0)node[below]{\small $\frac{2}{7}$}--(2/7,1);
                \draw[dotted](3/7,0)node[below]{\small $\frac{3}{7}$}--(3/7,1);
                \draw[dotted](4/7,0)node[below]{\small $\frac{4}{7}$}--(4/7,1);
                \draw[dotted](5/7,0)node[below]{\small $\frac{5}{7}$}--(5/7,1);
                \draw[dotted](6/7,0)node[below]{\small $\frac{6}{7}$}--(6/7,1);
                
		\end{tikzpicture}}
		
		\caption{The map $T_{\beta\circ n}$  on the left and $T_{n\circ \beta}$ on the right for $\beta=\frac{7}{3}$ and $n=3$. }
		\label{fig:mapscomposition}
\end{figure}

The branches of $T_{\beta\circ n}$ on the intervals $\left[\frac{k}{n}+\frac{l}{n\beta},\frac{k}{n}+\frac{l+1}{n\beta}\right),\  k=0,\ldots, n-1 \text{ and } l=0,\ldots, \lfloor \beta \rfloor -1$ are full, i.e. $T_{\beta\circ n}\left([\frac{k}{n}+\frac{l}{n\beta},\frac{k}{n}+\frac{l+1}{n\beta})\right)=[0,1)$, while on the intervals $\left[\frac{k}{n}+\frac{\lfloor \beta \rfloor }{n\beta},\frac{k+1}{n}\right) ,\  k=0,\ldots, n-1$ they are non-full with $T_{\beta\circ n}(\left[\frac{k}{n}+\frac{\lfloor \beta \rfloor }{n\beta},\frac{k+1}{n}\right)=[0,\beta-\lfloor \beta \rfloor)\neq [0,1)$.
In  \cite{DK10}, an exact formula for the invariant measure of such maps was derived. Before applying it to our case, we need to introduce a few things. A fundamental interval of rank $k$, is an interval of the form 
\[
I=I_1\cap T_{\beta\circ n}^{-1}I_2\cap \ldots \cap T_{\beta\circ n}^{-(k-1)}I_k, \text{ where } I_1,\ldots, I_k\in \mathcal{I}_{\beta\circ n}.
\]
A fundamental interval $I$ of rank $k$ is said to be full if $T_{\beta\circ n}^k(I)=[0,1)$ and non-full otherwise. Let $D_k$ be the collection of all non-full intervals of rank $k$ that are not subsets of full intervals of lower rank. Set $\phi_0(x)=1$ and 
\[
\phi_k(x)=\sum_{E\in D_k} \frac{1}{(\beta n)^k} \mathds{1}_{T^k_{\beta\circ n}(E))}(x),\quad k\geq 1.
\]
Let $\phi(x)=\sum_{k=0}^\infty \phi_k(x)$. By Theorem 3.2 in \cite{DK10}, the density $f_{\beta\circ n}$ of the $T_{\beta\circ n}$-invariant measure absolutely continuous with respect to Lebesgue measure $\lambda$, is given by
\begin{equation}\label{eq:invardensity}
    f_{\beta\circ n}(x)= \frac{\phi(x)}{\int \phi(x) d\lambda(x)}. 
\end{equation}

Note that $f_{\beta\circ n}$ is piecewise constant with discontinuity points occurring at points in the set
\[
O_{T_{\beta\circ n}}(1):=\{T_{\beta\circ n}^k(1):k\geq 1 \}, 
\]
where $T_{\beta\circ n}(1)=\beta-\lfloor \beta \rfloor$ (here we used $T_n(1):=\lim_{x\nearrow 1}T_n(x)=1$). We apply this result to derive the invariant measure of the composition $T_{\frac{p}{q}}\circ T_{kq}=T_{\frac{p}{q}\circ kq}$ with $p>q>1$ relatively prime, and $k\geq 1$.

\begin{proposition}\label{prop:explicit density}
    Let $p>q>1$ be relatively prime, and write $p=mq+r$ with $m\in\mathbb{N}$ and $1\leq r\leq q-1$. Then for any $k\geq 1$, the $T_{\frac{p}{q}\circ kq}$-invariant measure has density 
    \[
    f_{\frac{p}{q}\circ kq}(x)=\frac{p-r}{p}\left(1+\frac{q}{p-r}\mathds{1}_{[0,\frac{r}{q})}(x)  \right),
     \]
     which is independent of $k$.
\end{proposition}

\begin{proof}

We start with few observations. From $p=mq+r$, we get $\frac{p}{q}=m+\frac{r}{q}$. Since $\frac{r}{q}<1$, we have $\lfloor\frac{p}{q}\rfloor=m$. The map $T_{\frac{p}{q}\circ kq}$ is piecewise linear with underlying partition (fundamental intervals of rank 1) 
\[
\mathcal{I}_{\frac{p}{q}\circ kq}=\left\{\left[\frac{l}{kq}+\frac{j}{kp},\frac{l}{kq}+\frac{j+1}{kp}\right), \left[\frac{l}{kq}+\frac{m
}{kp},\frac{l+1}{kq}\right): l=0,\ldots, kq-1 \text{ and } j=0,\ldots, m-1\right\},
\]
see the bottom part of Figure \ref{fig:partitions}. 
All intervals of the form $\left[\frac{l}{kq}+\frac{j}{kp},\frac{l}{kq}+\frac{j+1}{kp}\right), l=0,\ldots kq-1 \text{ and } j=0,\ldots, m-1 $ are full while intervals of the form $\left[\frac{l}{kq}+\frac{m}{kp},\frac{l+1}{kq}\right), l=0,\ldots, kq-1 $ are non-full with $T_{\frac{p}{q}\circ kq}\left( \left[\frac{l}{kq}+\frac{m}{kp},\frac{l+1}{kq}\right)\right)=[0,\frac{r}{q}) $.

The set $D_1$ consists of $kq$ non-full intervals. Each of which is mapped to $[0,\frac{r}{q})$ and hence contains $rk$ non-full intervals of rank 2. Hence, $D_2$ consists of $rk^2q$ non-full intervals of rank 2. Each is mapped under $T_{\frac{p}{q}\circ kq}^2 $ to $[0,\frac{r}{q})$. Continuing this way, we see that $D_n$ consists of $r^{n-1}k^nq$ non-full intervals and each is mapped under $T_{\frac{p}{q}\circ kq}^n $ to  $[0,\frac{r}{q})$. Thus,
\[
\phi_n(x)=\frac{r^{n-1}k^nq}{(pk)^n}\mathds{1}_{[0,\frac{r}{q})}(x) =\frac{q}{p}\left(\frac{r}{p}\right)^{n-1} \mathds{1}_{[0,\frac{r}{q})}(x)
\]
which gives
\[
\phi(x)=\phi_0+\sum_{n=1}^\infty \phi_n(x)=1+\sum_{n=1}^\infty \frac{q}{p}\left(\frac{r}{p}\right)^{n-1} \mathds{1}_{[0,\frac{r}{q})}(x)= 1+\frac{q}{p}\frac{p}{p-r}\mathds{1}_{[0,\frac{r}{q})}(x)=1+\frac{q}{p-r}\mathds{1}_{[0,\frac{r}{q})}(x),
\]
and $\int \phi d\lambda(x)=1+\frac{q}{p-r}\frac{r}{q}=1+\frac{r}{p-r}=\frac{p}{p-r}$.
Therefore, 
\[
f_{\frac{p}{q}\circ kq}(x)=\frac{p}{p-r}\left(1+\frac{q}{p-r}\mathds{1}_{[0,\frac{r}{q})}(x) \right).
\]
\end{proof}

\begin{remark}\label{rem:otherdisc}
    Note that the invariant measure was found by using the orbit of 1. This is because all the other discontinuities map to the same point as 1. Here we use that $kq\in\mathbb{N}$. In general, for $T_{\beta_1\circ \beta_2}$, one should take into account not only the orbit of 1 but also the orbit of  $T_{\beta_1}(1)$.
\end{remark}

\section{Characterization of Coincidence of the measures}
In this section, we investigate the equality $\mu_{(\beta,n)}=\mu_{(\beta^\prime,m)}$ with $\beta,\beta^\prime>1$ non-integers and $n,m\geq 2$ integers. As a corollary of Proposition \ref{prop:explicit density}, we state our first result.

\begin{proposition}\label{prop:equal measure}
    Let $p>q>1$ be relatively prime integers. For any $k,l\geq 1$, we have $\mu_{(\frac{p}{q},kq)}=\mu_{(\frac{p}{q},lq)}$.
\end{proposition}

\begin{proof}
As we have seen earlier $\mu_{(\frac{p}{q},kq)}=\mu_{(\frac{p}{q},lq)}$ if and only if  $\mu_{ kq \circ \frac{p}{q}}=\mu_{ lq  \circ\frac{p}{q}} $ and $\mu_{\frac{p}{q}\circ kq}=\mu_{\frac{p}{q}\circ lq} $. Since $(\frac{p}{q})(kp)=kp$ and $(\frac{p}{q})(lp)=lp$ are both integers, we see that $T_{kq\circ \frac{p}{q}}(x)=kp\mod 1$ and $T_{lq\circ \frac{p}{q}}(x)=lp\mod 1$, and both preserve Lebesgue measure $\lambda$. From Proposition \ref{prop:explicit density}, the densities $f_{\frac{p}{q}\circ kq}(x)$ and $f_{\frac{p}{q}\circ lq}(x)$ are equal to 
\[
f_{\frac{p}{q}\circ kq}(x)=f_{\frac{p}{q}\circ lq}(x)=\frac{p}{p-r}\left(1+\frac{q}{p-r}\mathds{1}_{[0,\frac{r}{q})}(x) \right).
\]
Thus $\mu_{\frac{p}{q}\circ kq}(x)=\mu_{\frac{p}{q}\circ lq}(x)$ and $\mu_{ kq \circ \frac{p}{q}}=\mu_{ lq  \circ\frac{p}{q}}=\lambda $ and therefore,  $\mu_{(\frac{p}{q},kq)}=\mu_{(\frac{p}{q},lq)}$.
\end{proof}

\begin{theorem}\label{th:betarat}
    Let $\beta=\frac{p}{q}$ with $p>q>1$ relatively prime and $n\geq 2$ integer. Then for any $m\geq 2$ and $\beta^{\prime}>1$, we have
    $\mu_{(\beta,n)}=\mu_{(\beta^{\prime},m)}$ if and only if $\beta^\prime=\beta=\frac{p}{q}$ and $n,m\in\{kq:k\in\mathbb{N} \}$.
\end{theorem}

\begin{proof}
From Proposition \ref{prop:equal measure}, we see that $\mu_{(\frac{p}{q},kq)}=\mu_{(\frac{p}{q},lq)} $ for any $k,l\geq 1$. We prove the converse. Suppose $\mu_{(\frac{p}{q},n)}=\mu_{(\beta^\prime,m)}$. Then $\mu_{\frac{p}{q}\circ n}=\mu_{\beta^\prime\circ m}$ and $\mu_{n \circ\frac{p}{q}}=\mu_{ m\circ \beta^\prime}$. As was mentioned earlier $T_{n\circ \frac{p}{q}}(x)=T_{\frac{np}{q}}=\frac{np}{q}x \mod 1$ and $T_{m\circ \beta^\prime}(x)=T_{m\beta^\prime}=m\beta^\prime x \mod 1$. We first show that this implies that $\frac{np}{q}\in \mathbb{N}$. This is done by contradiction. Assume $\frac{np}{q}\notin \mathbb{N}$, then by the Theorem \ref{th:HuangWang} of Huang and Wang from \cite{HW25}, $\frac{np}{q}$ satisfies a quadratic equation of the form $x^2=ax+b$ with $a\geq b\geq 1$.
From the Rational root Theorem we get that, when we write $\frac{np}{q}=\frac{r}{s}$ in its simplest form, $s$ must be 1.
Thus, $\frac{np}{q}\in\mathbb{N}$, and since $p$ and $q$ are relatively prime, we have $n=kq$ for some $k\geq 1$. Since $T_{n\circ \frac{p}{q}}(x)=\frac{np}{q}x \mod 1$ and $\frac{np}{q}\in\mathbb{N}$, we see that the Lebesgue measure $\lambda$ is $T_{n\circ \frac{p}{q}}$-invariant, and hence $\lambda$ is also $T_{m\circ \beta}=T_{m\beta}$-invariant. This is only possible if $m\beta\in\mathbb{N}$. Thus $\beta=\frac{p^\prime}{q^\prime}\in\mathbb{Q}$ with $p^\prime>q^\prime\geq 1$ are relatively prime, and $m=lq^\prime$ for some $l\geq 1$. With this conclusion we turn to the second equality $\mu_{\frac{p}{q}\circ kq}=\mu_{\frac{p^\prime}{q^\prime}\circ lq^\prime}$. We write
$p=rq+t$ and $p^\prime=r^\prime q^\prime+t^\prime$, $r,r^\prime \geq 1, 1\leq t\leq q-1, 1\leq t^\prime\leq q^\prime-1$. From Proposition \ref{prop:explicit density}, the invariant densities of the maps $T_{\frac{p}{q}\circ kq}$ and $T_{\frac{p^\prime}{q^\prime}\circ lq^\prime}$ are given by \
\[
f_{\frac{p}{q}\circ kq}(x)=\frac{p-t}{p}\left(1+\frac{q}{p-t}\mathds{1}_{[0,\frac{t}{q})}(x) \right)=\frac{p-t}{p}+\frac{q}{p}\mathds{1}_{[0,\frac{t}{q})}(x),
\]
and
\[
f_{\frac{p^\prime}{q^\prime}\circ lq^\prime}(x)=\frac{p^\prime-t^\prime}{p^\prime}\left(1+\frac{q^\prime}{p^\prime-t^\prime}\mathds{1}_{[0,\frac{t^\prime}{q^\prime})}(x) \right)=\frac{p^\prime-t^\prime}{p^\prime}+\frac{q^\prime}{p^\prime}\mathds{1}_{[0,\frac{t^\prime}{q^\prime})}(x) .
\]
Since $f_{\frac{p}{q}\circ kq}(x)=f_{\frac{p^\prime}{q^\prime}\circ lq^\prime}(x)$ this gives us $\frac{q}{p}=\frac{q^\prime}{p^\prime}$ and since $p$ and $q$ are relatively prime and $p^\prime$ and $q^\prime$ as well, we find $p=p^\prime$ and $q=q^\prime$. 
Thus $\beta^\prime=\frac{p}{q}=\beta$, $n=kq$ and $m=lq$ as required. 
\end{proof}

We now deal with the case $\beta\notin\mathbb{Q}$. We do this in several steps.

\begin{theorem}\label{th:nequalm}
    Let $\beta>1$ be irrational and $n\geq 2$ an integer. Then, 
    for any $\beta^{\prime}>1$ with $\beta \neq \beta^\prime$, one has $\mu_{(\beta,n)}\not=\mu_{(\beta^\prime,n)}$.
\end{theorem}

\begin{proof}
Assume we can find a $\beta^\prime>1, \beta \neq \beta^\prime$ such that $\mu_{(\beta,n)}=\mu_{(\beta^\prime,n)}$. Then $ \mu_{\beta\circ n}=\mu_{\beta^\prime\circ n}$ and $\mu_{n\circ \beta}=\mu_{n\circ \beta^\prime}$. We start by studying the second equality. Since $T_{n\circ \beta}(x)=nx \,\mod1$ and $T_{n\circ \beta^\prime}(x)=n\beta^\prime x\,\mod1$,  then from Theorem \ref{th:HuangWang} of Huang and Wang from \cite{HW25}, we see that $n\beta$ satisfies $x^2=px+q$ with $p\geq q\ge1 $ and $n\beta^\prime=n\beta+1=n(\beta+\frac{1}{n})$. (Or with reverse roles but the proof is similar.) Thus $\beta^\prime=\beta+\frac{1}{n}$ and $n\beta=p+\frac{q}{n\beta}$. Then, $\lfloor n\beta \rfloor=p$, and since $\beta$ is irrational we have $\frac{p}{n}<\beta<\frac{p}{n}+\frac{1}{n}$. From the inequality $\beta>1$, we conclude that $\frac{p}{n}+\frac{1}{n}>1$ or that $n<p+1$. Hence $n\leq p$. If $p=n$, then $\lfloor \beta\rfloor=1=\frac{p}{n}=\lfloor \frac{p}{n}\rfloor$. If $n<p$, then we write $p=ns+t, s\geq 1$ and $0\leq t \leq n-1$. Thus $\frac{p}{n}=s+\frac{t}{n}$ and hence $s=\lfloor\frac{p}{n}\rfloor$. We now show that $\frac{p}{n}+\frac{1}{n}\leq \lfloor \frac{p}{n}\rfloor +1=s+1$. This is equivalent to $\frac{p}{n}-\lfloor\frac{p}{n} \rfloor \leq 1-\frac{1}{n}=\frac{n-1}{n}$, which is true since $\frac{p}{n}-\lfloor\frac{p}{n} \rfloor =\frac{t}{n}\leq \frac{n-1}{n} $. We now turn to the equality $ \mu_{\beta\circ n}=\mu_{\beta^\prime\circ n}$. From this we see that the densities $f_{\beta\circ n}$ and $f_{(\beta +\frac{1}{n})\circ n}$ must have the same discontinuities. To reach our contradiction we show that 
\[
T_{(\beta+\frac{1}{n})\circ n}(1)=\beta+\frac{1}{n}- \lfloor \beta+\frac{1}{n} \rfloor \notin O_{T_{\beta\circ n}}(1)=\Big\{T_{\beta\circ n}^k(1):k\geq 1 \Big\}
\]
with $T_{\beta\circ n}(1)=\beta-\lfloor \beta\rfloor$. We start by explicitly describing the set $O_{T_{\beta\circ n}}(1)$. For this we must consider several cases with parts.\\
\\
\underline{Case 1:} If $n=p$, then $\lfloor \frac{p}{n}\rfloor =1=\lfloor\beta \rfloor$ and $n^2\beta^2=n^2\beta+q$. Thus $n^2\beta(\beta-1)=q$ implying $\beta-1=\frac{q}{n^2\beta}=\frac{q/n}{n\beta}$.\\ (i) If $n|q$, then from $q\leq p=n$ we see that $n=q$. Then, $T_{\beta\circ n}(1)=\beta- \lfloor \beta \rfloor=\beta-1=\frac{1}{n\beta}$, which is an end point of the partition $\mathcal{I}_{\beta\circ n}$. Thus $T_{\beta\circ n}^k(1)=0$ for all $k\geq 2$ so that $O_{T_{\beta\circ n}}(1)=\Big\{0,\displaystyle\frac{1}{n\beta}\Big\}$.\\ (ii) If $n\not|q$, then from $q\leq p=n$, we get $q<n$ and $\frac{q}{n^2\beta}=\frac{q/n}{n\beta}<\frac{1}{n\beta}$. Thus, 
\[T^2_{\beta\circ n}(1)=n\beta (\frac{q}{n^2\beta})=\frac{q}{n},\]
which is an end-point. Hence,
$T_{\beta\circ n}^k(1)=0$ for all $k\geq 3$, and 
$O_{T_{\beta\circ n}}(1)=\Big\{0,\displaystyle\frac{q}{n^2\beta},\displaystyle\frac{q}{n}\Big\}$.\\
\\
\underline{Case 2:}
If $n<p$, then as above we write $p=ns+t$, $s\geq 1, 0\leq t\leq n-1$. So $\frac{p}{n}=s+\frac{t}{n}=\lfloor\frac{p}{n}\rfloor + \frac{t}{n}=\lfloor \beta \rfloor+ \frac{t}{n}$. Furthermore, from $n^2\beta^2=pn\beta+q$, we get $\beta=\frac{p}{n}+\frac{q}{n^2\beta}$.  Then, 
\[
T_{\beta\circ n}(1)=\beta-\left\lfloor\beta\right\rfloor=\beta -\left\lfloor\frac{p}{n}\right\rfloor=\left(\beta-\frac{p}{n}\right)+\left(\frac{p}{n}-\left\lfloor\frac{p}{n}\right\rfloor\right)=\frac{q}{n^2\beta}+\frac{t}{n}\in \begin{cases}
    \Big[\frac{t}{n}+\frac{\lfloor q/n\rfloor}{n\beta},\frac{t}{n}+\frac{\lfloor q/n\rfloor+1}{n\beta}\Big) & \text{ if } q<p\\
    \\
    \Big[\frac{t}{n}+\frac{ \lfloor \beta \rfloor}{n\beta},\frac{t+1}{n}\Big) & \text{ if } q=p
\end{cases}
\]
In both cases, $T_{\beta \circ n}^2(1)=\beta n \left(\frac{t}{n}+\frac{ q/n }{n\beta}\right)-\left(t\beta +\lfloor q/n \rfloor\right)= \frac{q}{n}-\left\lfloor \frac{q}{n} \right\rfloor $. We consider three cases. \\
(i) If $q=n<p$, then $\frac{q}{n}-\lfloor\frac{q}{n}\rfloor =0$ so that  $T_{\beta\circ n}^k(1)=0$ for all $k\geq 2$ and 
$O_{T_{\beta\circ n}}(1)=\Big\{0,\frac{t}{n}+\frac{1}{n\beta}\Big\}$.\\
(ii) If $q<n<p$, then $\frac{q}{n}-\lfloor\frac{q}{n}\rfloor =\frac{q}{n}$. Thus  $T_{\beta\circ n}^k(1)=0$ for all $k\geq 3$ and 
$O_{T_{\beta\circ n}}(1)=\Big\{0,\frac{q}{n},\frac{t}{n}+\frac{q}{n^2\beta}\Big\}$.\\
(iii) If $n<q\leq p$, then we write $q=nr+b$ with $r\geq 1$ and $0\leq b\leq n-1$. Then,  $\frac{q}{n}-\lfloor\frac{q}{n}\rfloor =\frac{b}{n}$ which is an end-point. Therefore  $T_{\beta\circ n}^k(1)=0$ for all $k\geq 3$ and 
$O_{T_{\beta\circ n}}(1)=\Big\{0,\frac{b}{n},\frac{t}{n}+\frac{q}{n^2\beta}\Big\}$. Summarizing, we have
\[O_{T_{\beta\circ n}}(1)=
\begin{cases}
\Big\{0,\displaystyle\frac{1}{n\beta}\Big\}
& \text{ if } q=n=p\\
\\
\Big\{0,\displaystyle\frac{q}{n^2\beta},\displaystyle\frac{q}{n}\Big\} & \text{ if } q<n=p\\
\\
\Big\{0,\frac{t}{n}+\frac{1}{n\beta}\Big\} & \text{ if } q=n<p,\, p=ns+t\\
\\
\Big\{0,\frac{q}{n},\frac{t}{n}+\frac{q}{n^2\beta}\Big\} & \text{ if } q<n<p,\, p=ns+t\\
\\
\Big\{0,\frac{b}{n},\frac{t}{n}+\frac{q}{n^2\beta}\Big\} & \text{ if } n<q\leq p,\, p=ns+t,\,q=nr+b.
\end{cases}
\]
\\In all cases $O_{T_{\beta\circ n}}(1)$ is a finite set consisting of at most three elements. \\

We now calculate $T_{(\beta +\frac{1}{n})\circ n}(1)=\beta +\frac{1}{n}-\lfloor \beta +\frac{1}{n} \rfloor$ in all cases considered. We first show that
\[\lfloor \beta +\frac{1}{n} \rfloor\in \Big\{\lfloor\beta \rfloor, \lfloor \beta \rfloor +1 \Big\}=\Big\{\lfloor\frac{p}{n} \rfloor, \lfloor \frac{p}{n} \rfloor +1 \Big\}.\] 
If $p=n$, then $1<\beta <1+\frac{1}{n}$ and $1<1+\frac{1}{n}<1+\frac{2}{n}\leq 2$ (equality holds if $n=2$). Thus, $\lfloor \beta +\frac{1}{n}\rfloor=1=\lfloor \beta \rfloor$.
If $n<p$, then $\frac{p}{n}+\frac{1}{n}<\beta+\frac{1}{n} <\frac{p}{n}+\frac{2}{n}$. From above, we have $\frac{p}{n}+\frac{1}{n}\leq \lfloor \frac{p}{n} \rfloor +1$ and $\frac{p}{n}+\frac{2}{n}<\lfloor \frac{p}{n} \rfloor+2$. Thus $\beta +\frac{1}{n}\in \Big(\lfloor\frac{p}{n}\rfloor,\lfloor\frac{p}{n}\rfloor+2\Big)$ implying that $\lfloor\beta +\frac{1}{n}\rfloor \in \Big\{\lfloor\frac{p}{n}\rfloor,\lfloor\frac{p}{n}\rfloor+1\Big\}$. \\
\\
We start by assuming $\lfloor\beta +\frac{1}{n}\rfloor=\lfloor\frac{p}{n}\rfloor=\lfloor\beta\rfloor$, and calculate $T_{(\beta+\frac{1}{n})\circ n}(1)$. Note that 
\[T_{(\beta+\frac{1}{n})\circ n}(1)=\beta+\frac{1}{n}-\lfloor\beta+\frac{1}{n} \rfloor=\beta+\frac{1}{n}-\lfloor\beta \rfloor=T_{\beta\circ n}(1)+\frac{1}{n}.\]
Then
 \[
 T_{(\beta+\frac{1}{n})\circ n}(1)=\begin{cases}
     \frac{1}{n}+\frac{1}{n\beta} & \text{ if } q=n=p\\
     \\
     \frac{1}{n}+\frac{q}{n^2\beta} & \text{ if } q<n=p\\
     \\
     \frac{t+1}{n}+\frac{q}{n^2\beta} & \text{ if } n< q\le p,\, p=ns+t\\
 \end{cases}
 \]
which in all cases is not an element of $O_{T_{\beta\circ n}}(1)$. Thus, $f_{\beta\circ n}$ and $f_{(\beta +\frac{1}{n})\circ n}$ have different discontinuity points and hence $\mu_{(\beta,n)}\neq  \mu_{(\beta^\prime,n)}$, a contradiction.\\ 
\\
We now consider the case $\lfloor\beta+\frac{1}{n} \rfloor=\lfloor\beta \rfloor+1 $. We have seen that if $p=n$, then $\lfloor\beta+\frac{1}{n} \rfloor =\lfloor\beta \rfloor=1$ so we must have $n<p$. We write, as before, $p=ns+t$, $s\geq 1, 0\leq t\leq n-1$ which gives $\frac{p}{n}-\lfloor\frac{p}{n} \rfloor=\frac{t}{n} $. Since $\beta<\frac{p}{n}+\frac{1}{n}$, the only way $\beta +\frac{1}{n}\geq \lfloor \beta \rfloor +1=\lfloor \frac{p}{n} \rfloor +1$ is that $t=n-1$, see also Figure \ref{fig:estimate}.\\
\begin{figure}[thb]
\center{\begin{tikzpicture}[xscale=1]
\draw[ >=triangle 45, <->] (0,0.5) -- (2,0.5)node[above,midway]{$\frac{n-1}{n}$};

\draw(0,0)--(4,0);
\draw(0,-0.1)node[below]{$\lfloor\frac{p}{n}\rfloor$}--(0,0.1);

\draw(3,-0.1)node[below]{$\lfloor\frac{p}{n}\rfloor+1$}--(3,0.1);

\draw(2,-0.1)node[below]{$\frac{p}{n}$}--(2,0.1);
\draw(3.6,0)circle(1.75pt)[color=red!50,fill=red!50, fill opacity=1];
\draw(3.6,0)node[above]{$\beta+\frac{1}{n}$};
\draw(2.7,0)circle(1.75pt)[color=red!50,fill=red!50, fill opacity=1];
\draw(2.7,0)node[above]{$\beta$};

\end{tikzpicture}}
\caption{}
\label{fig:estimate}
\end{figure}

Since $n<p$ and $t=n-1$, we have 
\begin{align*}
T_{(\beta+\frac{1}{n})\circ n}(1)&= \beta+\frac{1}{n}-\lfloor\beta+\frac{1}{n}\rfloor\\
                                 &= \beta-\lfloor\beta\rfloor+\frac{1}{n}-1\\
                                 &= \frac{q}{n^2\beta}+\frac{n-1}{n}+\frac{1}{n}-1\\
                                 &= \frac{q}{n^2\beta} \notin O_{T_{\beta\circ n}}(1).
\end{align*}
Leading to a contradiction. Thus  $\mu_{(\beta,n)}\neq  \mu_{(\beta^\prime,n)}$ for any $\beta^\prime>1$ with $\beta\not= \beta^{\prime}$.
\end{proof}

\begin{theorem}\label{th:betairr}
    Let $\beta>1$ be irrational  and $n,m\geq 2$. Then, for any irrational $\beta^\prime>1$ with $\beta^{\prime}\not= \beta$, we have $\mu_{(\beta,n)}\neq  \mu_{(\beta^\prime,m)}$.
\end{theorem}

\begin{proof}
If $n=m$, then Theorem \ref{th:nequalm} gives the result. Assume $n\neq m$ and $\mu_{(\beta,n)}=\mu_{(\beta^\prime,m)}$ for some irrational $\beta^\prime>1$. Then, $ \mu_{\beta\circ n}=\mu_{\beta^\prime\circ m}$ and $\mu_{n\circ \beta}=\mu_{m\circ \beta^\prime}$.
Since $T_{n\circ \beta}(x)=n\beta x \mod 1$ and $T_{m\circ \beta^\prime}(x)=m\beta^\prime x \mod 1$, by Theorem \ref{th:HuangWang} of Huang and Wang of \cite{HW25}, $n\beta$ must satisfy a quadratic equation of the form $x^2=px+q$ with $p\geq q\geq 1$ and $m\beta^\prime=n\beta +1=n(\beta+\frac{1}{n})$, or with reverse roles, but the proof is similar. Now $\beta^\prime=\frac{n}{m}(\beta+\frac{1}{n})$. Since $\beta^\prime>1$, we get $\frac{n}{m}(\beta+\frac{1}{n})>1$ or that $n\beta+1>m$. From $n^2\beta^2=pn\beta+q$, we get $n\beta=p+\frac{q}{n^2\beta}$. Thus, $m\beta^\prime=n\beta+1=p+1+\frac{q}{n^2\beta}$. Since $p<n\beta<p+1$, we get $\frac{p}{n}<\beta<\frac{p}{n}+\frac{1}{n}$ and $p+1<n\beta+1<p+2$. Using the fact $\beta>1$, the first inequality gives $n\leq p$. Using $n\beta+1>m$, we see that $m\leq p+1$. From the equality $f_{\beta\circ n}=f_{\beta^\prime\circ m}$, we conclude that  $O_{T_{\beta\circ n}}(1)=O_{T_{\beta^\prime\circ m}}(1)=O_{T_{\frac{n}{m}(\beta+\frac{1}{n})\circ m}}(1)$. 
The set $O_{T_{\beta\circ n}}(1) $ has been calculated in Theorem \ref{th:nequalm}. To reach our contradiction, it is enough to show 
\[
T_{\frac{n}{m}(\beta+\frac{1}{n})\circ m}(1)=\frac{n}{m}(\beta+\frac{1}{n})-\lfloor\frac{n}{m}(\beta+\frac{1}{n}) \rfloor=\frac{n}{m}\beta+\frac{1}{m}-\lfloor\frac{n}{m}\beta+\frac{1}{m} \rfloor\notin O_{T_{\beta\circ n}}(1).
\]
From $p<n\beta<p+1$, we get $\frac{p+1}{m}<\frac{n}{m}\beta +\frac{1}{m}<\frac{p+2}{m}$ with $m\leq p+1$. We consider two cases: \\
(i) If $m=p+1$, then the above gives $1<\frac{n}{m}\beta + \frac{1}{m}<1+\frac{1}{m}$. Thus, $\lfloor\frac{n}{m}\beta +\frac{1}{m}\rfloor=1$, and 
\begin{align*}
    T_{\frac{n}{m}(\beta+\frac{1}{n})\circ m}(1)&=T_{(\frac{n}{p+1}\beta+\frac{1}{p+1})\circ p+1}(1)\\
                                                &=\frac{n}{p+1}\beta +\frac{1}{p+1}- \lfloor\frac{n}{p+1}\beta +\frac{1}{p+1} \rfloor\\
                                                &=\frac{p}{p+1}+\frac{q}{n(p+1)\beta}+\frac{1}{p+1}-1\\
                                                &=\frac{q}{n(p+1)\beta}\notin O_{T_{\beta\circ n}}(1),
\end{align*}
where we used $n\beta=p+\frac{q}{n\beta}$.\\
(ii) If $m<p+1$, then we write $p+1=mu+v$, $u\geq1, 0\leq v\leq m-1$. Then, $u=\lfloor \frac{p+1}{m}\rfloor$, and since $\frac{p+1}{m}<\frac{n}{m}\beta +\frac{1}{m}<\frac{p+2}{m}$, we get $u+\frac{v}{m}<\frac{n}{m}\beta +\frac{1}{m}<u+\frac{v+1}{m}$. So $\lfloor \frac{n}{m}\beta +\frac{1}{m}  \rfloor=u= \lfloor \frac{p+1}{m} \rfloor$. Thus,
\begin{align*}
    T_{\frac{n}{m}(\beta+\frac{1}{n})\circ m}(1)&= \frac{n}{m}\beta +\frac{1}{m}-\lfloor \frac{n}{m}\beta +\frac{1}{m} \rfloor\\
                                                &= \frac{n}{m}\beta +\frac{1}{m}-\lfloor \frac{p+1}{m} \rfloor\\
                                                &= \frac{p}{m}+\frac{q}{nm\beta}+\frac{1}{m}-\lfloor \frac{p+1}{m} \rfloor\\
                                                &=\frac{p+1}{m}-\lfloor \frac{p+1}{m} \rfloor+\frac{q}{nm\beta}\\
                                                &=\frac{v}{m}+\frac{q}{nm\beta}\notin O_{T_{\beta\circ n}}(1) 
\end{align*}
where we used $n\beta=p+\frac{q}{n\beta}$ and the last equation follows from the fact that $n\neq m$. Thus, $f_{\beta\circ n}\neq f_{\beta^\prime\circ m}$ leading to a contradiction. Therefore, $\mu_{(\beta,n)}\neq \mu_{(\beta^\prime,m)}$ for any $\beta^\prime >1$.

\end{proof}
\begin{remark}\label{re:difint}
The proof of the above theorem also shows that if $m\not= n$, then for any irrational $\beta>1$,
 $\mu_{(\beta,n)}\not= \mu_{(\beta,m)}$. For if these measures are equal, then by the result of Huang and Wang we must have $m\beta=n\beta+1$ leading to $\beta$ being rational, which is a contradiction.
   
\end{remark}

As a direct corollary of Theorems \ref{th:betarat}, \ref{th:betairr} and Remark 
\ref{re:difint}, we have the following results.
\begin{theorem}
Let $\beta,\beta^\prime>1$ be non-integer real numbers and $n,m\geq 2$, integers. Then, 
    $\mu_{(\beta,n)}=\mu_{(\beta^\prime,m)}$ if and only if $\beta=\beta^\prime =\frac{p}{q}$ with $p$ and $q$ relatively prime and $n,m\in\{kq: k\geq 1\}$.
\end{theorem}

\begin{theorem}
Let $m\ge 2$ and $n\ge 2$ be relatively prime integers. Then, for any two non-integer real numbers $\beta,\,\beta^{\prime}>1$ one has $\mu_{(\beta,n)}\not= \mu_{(\beta^{\prime}, m)}$.    
\end{theorem}

\section{Preliminary Investigations, when both bases are non-integer}

In this last section, we shed some light onto the case where both bases are rational but not integer. As we shall see, things can get complicated quickly.  Even in the case where orbit sets are equal, we can find different measures.

\begin{lemma}\label{lem:denomblowup}
    Let $\beta_1=\frac{p_1}{q_1}$ and $\beta_2=\frac{p_2}{q_2}$ if $p_1\neq k q_2$ for any $k\in\mathbb{N}$ then   $O_{T_{\beta_1\circ \beta_2}}(1)$ is an infinite set. Moreover, define $T_{\beta_1\circ \beta_2}^n(1)=\frac{t_n}{s_n}$ and let us write $\frac{p_1}{q_2}=\frac{p}{z}$ such that $\gcd (p,z)=1 $. \\
    Then we can write $s_n=m_nz^n$ with $m_n$ a positive sequence of integers. We also have $\gcd(t_n,z)=1$.
\end{lemma}

\begin{proof}
    We prove the statement of $\frac{t_n}{s_n}$ by induction. The fact that $O_{T_{\beta_1\circ \beta_2}}(1)$ is an infinite set is then immediate. Let us write $p_1=rp$ and $q_2=rz$ for some $r\in\mathbb{N}$. First note that in general, for some $k,l\in\mathbb{N}$, we have
\begin{align*}
    T_{\beta_1\circ \beta_2}(x)&=\beta_1\beta_2x-(k\beta_1+l)\\
    &=\frac{p_1p_2}{q_1q_2}x-\frac{kp_1+lq_1}{q_1}\\
    &=\frac{p_1p_2x-q_2(kp_1+lq_1)}{q_1q_2}\\
    &=\frac{rpp_2x-rz(kp_1+q_1)}{q_1rz}=\frac{pp_2x-z(kp_1+q_1)}{q_1z}
\end{align*}

From this we get $T_{\beta_1\circ \beta_2}(1)=\frac{pp_2-z(kp_1+q_1)}{q_1z}=\frac{t_1}{s_1}$. We have $\gcd(p,z)=1$ and since $\gcd(p_2,q_2)=1$ we also have $\gcd(p_2,z)=1$. Furthermore, $z(kp_1+q_1)$ is a multiple of $z$. We find that $\gcd(t_1,z)=1$ and $s_1=m_1z$ for some $m_1\in\mathbb{N}$. Now, suppose it holds for $n$. For some $k,l\in\mathbb{N}$ we have
\[
T_{\beta_1\circ \beta_2}\left(\frac{t_n}{s_n}\right)=\frac{pp_2t_n-zs_n(kp_1+q_1)}{m_nz^nq_1z}=\frac{t_{n+1}}{s_{n+1}}.
\]
Similarly as before $\gcd(p,z)=1$ and $\gcd(p_2,z)=1$, but now also $\gcd(t_n,z)=1$, and $zs_n(kp_1+q_1)$ is a multitude of $z$. We find that $\gcd(t_{n+1},z)=1$ and $s_{n+1}=m_{n+1}z^{n+1}$ for some $m_{n+1}\in\mathbb{N}$.

\end{proof}
Note that  in Lemma~\ref{lem:denomblowup} it is possible that $\frac{t_n}{s_n}$, when writing  $s_n=m_nz^n$  for some positive sequence $m_n$, is still not in the most reduced form. Though, the lemma as stated suffices for our purposes. Also note that any rational number $\frac{t_1}{s_1}$, when writing $\frac{t_1}{q_2}=\frac{p}{z}$ such that $\gcd(p,z)=1$ which has $z>1$, the same argument works and the orbit of it will be infinite. In particular, for the orbit of the other discontinuity which is relevant for the invariant measure (see Remark~\ref{rem:otherdisc} ). Let us write $p_1=dq_1+j$, then the other discontinuity point is $\frac{j}{q_1}$ which will have an infinite orbit when $j\neq kq_2$ for any $k\in \mathbb{N}$. This does not have to be the case. For example, for $\beta_1=\frac{7}{4}$ and $\beta_2=\frac{5}{3}$, the orbit of $1$ is infinite while  $T_{\beta_1}(1)=\frac{3}{4}$ and $T_{\beta_1\circ\beta_2}(\frac{3}{4})=0$. Instead of trying to find out when it is difficult to find equal measures let us now look what happens if we are in a relatively good setting. To that extent, let us investigate the case that $\beta_1=\frac{p_1}{q_1}$ and $\beta_2=\frac{p_2}{q_2}$ where $p_1=k_1q_2$ and $p_2=k_2q_1$ for some $k_1,k_2\in \mathbb{N}$. Here we have the following lemma. 

\begin{lemma}
    Let $\beta_1=\frac{p_1}{q_1}$ and $\beta_2=\frac{p_2}{q_2}$ where $p_1=k_1q_2$ and $p_2=k_2q_1$ for some $k_1,k_2\in \mathbb{N}$. Furthermore, write $p_1=d_1q_1+j_1$ and $p_2=d_2q_2+j_2$.  
    Then all of the following sets are finite: $ O_{T_{\beta_1\circ \beta_2}}(1)$, and $ O_{T_{\beta_2\circ \beta_1}}(1)$, and  $O_{T_{\beta_1\circ \beta_2}}(\frac{j_1}{q_1})$, and $ O_{T_{\beta_2\circ \beta_1}}(\frac{j_2}{q_2})$.
\end{lemma}

\begin{proof}
Let us calculate  
\[
T_{\beta_1\circ \beta_2}(1)=\frac{p_1p_2-q_2(kp_1+lq_1)}{q_1q_2}=\frac{k_1k_2q_1q_2-q_2(kk_1q_2+lq_1)}{q_1q_2}=\frac{k_1k_2q_1-(kk_1q_2+lq_1)}{q_1}.
\]
Now, any rational that can be written as $\frac{p}{q_1}$ is mapped to a rational with the same denominator 
\[
T_{\beta_1\circ \beta_2}\left(\frac{p}{q_1}\right)=\frac{p_1p_2\frac{p}{q_1}-q_2(kp_1+lq_1)}{q_1q_2}=\frac{k_1k_2q_1q_2p-q_1q_2(kp_1+lq_1)}{q_1^2q_2}=\frac{k_1k_2p-(kp_1+lq_1)}{q_1}.
\]
From this we conclude that $O_{T_{\beta_1\circ \beta_2}}(1)$ and $O_{T_{\beta_1\circ \beta_2}}(\frac{j_1}{q_1})$ are finite sets. For $O_{T_{\beta_2\circ \beta_1}}(1)$ and $O_{T_{\beta_1\circ \beta_2}}(\frac{j_2}{q_2})$, the roles of $\beta_1$ and $\beta_2$ are reversed.
\end{proof}

The setting of the previous lemma seems promising, but as the following example shows, orbit sets of the discontinuities can be equal but the weights in the formula for the invariant density can differ.
\begin{example}\label{ex:nonequalmeasures}
    Let $\beta_1=\frac{3}{2}$ and $\beta_2=\frac{4}{3}$ and  $\beta_3=\frac{9}{2}$ and $\beta_4=\frac{4}{3}$. 
    Then the discontinuities of the invariant densities of $T_{\beta_1\circ \beta_2}$ and of $T_{\beta_3\circ \beta_4}$ are the same but  the heights are different. 
\end{example}
We will show this now (for the graphs of the maps, see Figure \ref{fig:examplenonequalmeasures}). We have $T_{\beta_1\circ\beta_2}(1)=T_{\beta_3\circ\beta_4}(1)=\frac{1}{2}$ and $T_{\beta_1\circ\beta_2}(\frac{1}{2})=T_{\beta_3\circ\beta_4}(\frac{1}{2})=0$. The other discontinuity is $\frac{1}{2}$ for both, for which we just  calculated is mapped to $0$. This gives us that both densities are of the form $C(1+A\cdot\mathds{1}_{[0,\frac{1}{2}]})$. One can check that for $T_{\beta_1\circ\beta_2}$ the invariant density is given by $\frac{2}{3}(1+\mathds{1}_{[0,\frac{1}{2}]})$. For $T_{\beta_3\circ\beta_4}$ the density is given by $\frac{6}{7}(1+\frac{1}{3}\mathds{1}_{[0,\frac{1}{2}]})$. This can either be found by constructing some equalities using the invariance of the measure or by using the explicit form of the measure from \cite{DK10} that is given in Equation (\ref{eq:invardensity}).

\begin{figure}[ht]
		\centering
		
		\subfigure{\begin{tikzpicture}[scale=5]
				\draw[white] (-0.25,0)--(1,0);
				\draw(0,0)node[below]{\small $0$}--(1,0)node[below]{\small $1$}--(1,1)--(0,1)node[left]{\small $1$}--(0,0);

				\draw[thick, blue, smooth, samples =20, domain=0:1/2] plot(\x,{2*\x)});
				\draw[thick,blue, smooth, samples =20, domain=1/2:3/4] plot(\x,{2*\x-1)});
                \draw[thick,blue, smooth, samples =20, domain=3/4:1] plot(\x,{2*\x-3/2)});

				\draw[dotted](1/2,0)node[below]{\small $\frac{1}{2}$}--(1/2,1);
                \draw[dotted](0,0.5)node[left]{\small $\frac{1}{2}$}--(1,0.5);
                \draw[dotted](0.75,0)node[below]{\small $\frac{3}{4}$}--(0.75,1);

		\end{tikzpicture}}
        \subfigure{\begin{tikzpicture}[scale=5]
				\draw[white] (-0.25,0)--(1,0);
				\draw(0,0)node[below]{\small $0$}--(1,0)node[below]{\small $1$}--(1,1)--(0,1)node[left]{\small $1$}--(0,0);

				\draw[thick, blue, smooth, samples =20, domain=0:1/6] plot(\x,{6*\x)});
				\draw[thick, blue, smooth, samples =20, domain=1/6:2/6] plot(\x,{6*\x-1)});
				\draw[thick, blue, smooth, samples =20, domain=2/6:3/6] plot(\x,{6*\x-2)});
                \draw[thick, blue, smooth, samples =20, domain=3/6:4/6] plot(\x,{6*\x-3)});
                \draw[thick, blue, smooth, samples =20, domain=4/6:3/4] plot(\x,{6*\x-4)});
                \draw[thick, blue, smooth, samples =20, domain=3/4:11/12] plot(\x,{6*\x-9/2)});
                \draw[thick, blue, smooth, samples =20, domain=11/12:1] plot(\x,{6*\x-11/2)});
				\draw[dotted](1/2,0)node[below]{\small $\frac{1}{2}$}--(1/2,1);

                \draw[dotted](0,0.5)node[left]{\small $\frac{1}{2}$}--(1,0.5);

                \draw[dotted](1/6,0)node[below]{\small $\frac{1}{6}$}--(1/6,1);
                 \draw[dotted](1/3,0)node[below]{\small $\frac{1}{3}$}--(1/3,1);
                 \draw[dotted](2/3,0)node[below]{\small $\frac{2}{3}$}--(2/3,1);
                 \draw[dotted](3/4,0)node[below]{\small $\frac{3}{4}$}--(3/4,1);
                  \draw[dotted](11/12,0)node[below]{\small $\frac{11}{12}$}--(11/12,1);

		\end{tikzpicture}}
		
		\caption{The map $T_{\beta_1}\circ T_{\beta_2}$  on the left and $T_{\beta_3}\circ T_{\beta_4}$ on the right for Example \ref{ex:nonequalmeasures}. }
		\label{fig:examplenonequalmeasures}
\end{figure}
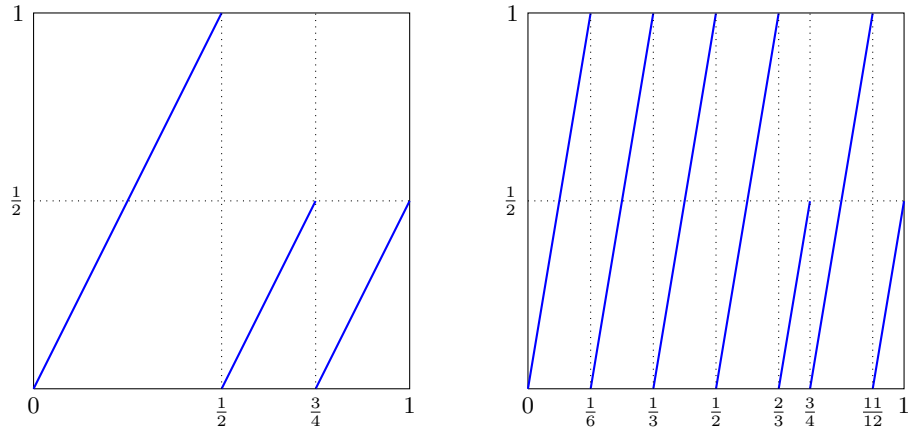

It seems that coincidence of invariant measure does not occur often. We believe that, in fact, we found all of them.

\begin{conjecture}
    All pairs of pairs $(\beta_1,\beta_2),(\beta_3,\beta_4)\in\mathbb{R}^2_{>1}$ for which we have $\mu_{(\beta_1,\beta_2)}=\mu_{(\beta_3,\beta_4)}$ are characterised by Theorem~\ref{th:maintheorem}.
\end{conjecture}

\bibliographystyle{alpha}
\bibliography{biblio}

\end{document}